\documentclass[10pt]{amsart}
\usepackage{enumerate,amsmath,amssymb,latexsym,
amsfonts, amsthm, amscd, MnSymbol}

\theoremstyle{plain}

\newtheorem{theorem}{Theorem}

\numberwithin{equation}{section}

\newcommand{\ra}{\rightarrow}

\addtocounter{theorem}{-1}

\begin{document}

\title {Julia operators and Halmos dilations}

\date{}

\author[P.L. Robinson]{P.L. Robinson}

\address{Department of Mathematics \\ University of Florida \\ Gainesville FL 32611  USA }

\email[]{paulr@ufl.edu}

 \keywords{}

\begin{abstract}

We offer a simple direct proof of the unitarity of the Julia operator associated to a contraction $A$, from which follow the intertwining identity $(I - A A^*)^{1/2} A = A (I - A^* A)^{1/2}$ and the unitarity of Halmos dilations. 

\end{abstract}

\maketitle

\medbreak

Let $A : \mathbb{K} \ra \mathbb{H}$ be a contraction from the complex Hilbert space $\mathbb{K}$ to the complex Hilbert space $\mathbb{H}$. The associated {\it Julia operator} is the unitary operator $J_A \in B(\mathbb{H} \oplus \mathbb{K})$ having $2 \times 2$ block form 

\[ J_A \; = \; 
\begin{bmatrix}
 (I - A A^*)^{1/2} & A \\
 - A^* & (I - A^* A)^{1/2}
\end{bmatrix} 
\]
\medbreak 
\medbreak
\noindent
where $(I - A A^*)^{1/2} \in B(\mathbb{H})$ and $(I - A^* A)^{1/2} \in B(\mathbb{K})$ denote the positive square-roots as usual. This operator is named in recognition of Gaston Julia [3] and features in an operator-theoretic M\"obius transformation approach [10] to the theorem of Parrott [4] on contractive completions of partially-filled block operators. 

\medbreak 

Our primary aim here is extremely modest: to offer a simple direct proof of the fact that $J_A$ is indeed unitary. As immediate consequences, we deduce a simplified proof of the standard intertwining identity 
$$(I - A A^*)^{1/2} A = A (I - A^* A)^{1/2}$$
\noindent
and a simplified proof of the fact that the Halmos [1] dilation associated to a contractive endomorphism of the single Hilbert space $\mathbb{H}$ is unitary. Our simplified proofs make effective use of the $2 \times 2$ block-operator context in which Julia operators and Halmos dilations arise. 

\medbreak 

\begin{theorem} \label{Julia}
If $A : \mathbb{K} \ra \mathbb{H}$ is a contraction, then $J_A \in B(\mathbb{H} \oplus \mathbb{K})$ is unitary. 
\end{theorem} 

\begin{proof} 
Introduce on $\mathbb{H} \oplus \mathbb{K}$ the skew-adjoint operator $C$ with block form 

\[
\begin{bmatrix}
0 & A \\
 - A^* & 0
\end{bmatrix} 
\]
\medbreak 
\medbreak 
\noindent 
and the positive operator $D$ with block form 

\[
\begin{bmatrix}
 (I - A A^*)^{1/2} & 0 \\
 0 & (I - A^* A)^{1/2}
\end{bmatrix} 
\]
\medbreak 
\medbreak 
\noindent 
so that 
$$J_A = D + C.$$
\medbreak 
\noindent 
The operator $C$ commutes with $D^2$ and hence commutes with its positive square-root $D$. Thus  
$$J_A^* J_A = (D - C) (D + C) = D^2 - C^2 = I$$
and $I = J_A J_A^*$ likewise, so $J_A$ is unitary as claimed. 
\end{proof} 

\medbreak 

In some presentations, the Julia operator is defined with columns switched; this produces an operator from $\mathbb{K} \oplus \mathbb{H}$ to $\mathbb{H} \oplus \mathbb{K}$ and thereby obstructs direct application of the square-root argument presented here. 

\medbreak 

As a first corollary, we deduce at once the standard intertwining identity. 

\medbreak 

\begin{theorem} \label{inter}
If $A : \mathbb{K} \ra \mathbb{H}$ is a contraction, then
$$(I - A A^*)^{1/2} A = A (I - A^* A)^{1/2}.$$
\end{theorem} 

\begin{proof} 
Simply compare off-diagonal blocks in the commutative identity $D C = C D$ of the proof for Theorem \ref{Julia}. 
\end{proof} 

As a second corollary, we deduce at once the unitarity of the Halmos dilation. 

\medbreak 

\begin{theorem} \label{Halmos}
If $A \in B(\mathbb{H})$ is a contraction, then its Halmos dilation 
\[
\begin{bmatrix}
A &  (I - A A^*)^{1/2} \\
 (I - A^* A)^{1/2} & -A^*
\end{bmatrix} \; \in \; B(\mathbb{H} \oplus \mathbb{H})
\]
\medbreak 
\noindent 
is unitary. 
\end{theorem} 

\begin{proof} 
Simply note that  if $F \in B(\mathbb{H} \oplus \mathbb{H})$ is the (unitary) `flip' operator with block form 
\[
\begin{bmatrix}
 0 &  I \\
I & 0
\end{bmatrix} 
\]
\medbreak 
\noindent 
then the indicated Halmos dilation is precisely the composite $J_A F$. 
\end{proof} 

\medbreak 

We remark that historically, the traditional approach to these results has been to start from the standard intertwining identity of Theorem \ref{inter}, whence Theorem \ref{Julia} and Theorem \ref{Halmos} follow by matrix multiplication. Write $S = (I - A A^*)^{1/2}$ and $T = (I - A^* A)^{1/2}$ for convenience. Halmos [1] observes that 
$$S^2 A = (I - A A^*) A = A (I - A^* A) = A T^2$$
so that (by induction and linearity) $p(S^2) A = A p(T^2)$ when $p$ is any polynomial and therefore (by the Weierstrass approximation theorem) $f(S^2) A = A f(T^2)$ when $f$ is any continuous function; the case $f: [0, 1] \ra [0, 1] : t \mapsto t^{1/2}$ yields the standard intertwining identity. This very same approach is taken by Halmos in Problem 222 of {\it A Hilbert Space Problem Book} [2]. According to Sz.-Nagy [9], each contraction on $\mathbb{H}$ has a unitary {\it power} dilation. The proof of this fact presented in [5] makes use of the intertwining identity, following exactly the traditional justification due to Halmos; see page 467. The simplified construction of a unitary power dilation by Sch\"affer [8] again rests on this traditional justification. Sarason [6] surveys all of this and more, the traditional justification coming on page 196. The traditional justification also supports the theorem of Parrott [4] on contractive completions of partially-filled $2 \times 2$ block operators; combine (i) at the top of page 313 with the calculation at the top of page 316. Young [10] presents an alternative approach to the Parrott theorem, based on operator-theoretic M\"obius transformations; see Theorem 12.20 for the traditional justification there. The intertwining identity is important in the model theory originating with de Branges and Rovnyak: for example, it can be found on page 3 of [7], yet again with the traditional justification. This list of references involving the traditional justification is a mere sampling; it could be lengthened considerably.    

\medbreak 

The ubiquitous traditional justification of the intertwining identity essentially recapitulates the standard procedure whereby a positive operator $R$ is shown to have a unique positive square-root $R^{1/2}$, which commutes with every operator that commutes with $R$ itself. Our simplified justification forgoes this recapitulation, instead appealing directly to the square-root itself. By first proving Theorem \ref{Julia} we make entirely natural use of the $2 \times 2$ block-operator setting of the theory. In hindsight and in spirit, our simplified approach is thus kin to the Berberian route from the Fuglede theorem to its Putnam extension, which asserts that an intertwiner of two normal operators likewise intertwines their adjoints.  
\medbreak

\bigbreak

\begin{center} 
{\small R}{\footnotesize EFERENCES}
\end{center} 
\medbreak 

[1] P.R. Halmos, {\it Normal dilations and extensions of operators}, Summa Brasiliensis Mathematicae Vol. II, Ano VI, 125-134 (1950). 

\medbreak 

[2] P.R. Halmos, {\it A Hilbert Space Problem Book}, Graduate Texts in Mathematics {\bf 19}, Second Edition, Springer-Verlag (1982). 

\medbreak 

[3]  G. Julia, {\it Les projections des syst\`emes orthonormaux de l'espace hilbertien et les op\'erateurs born\'es}, Comptes Rendus de l'Academie des Sciences {\bf 219}, 8-11 (1944).  

\medbreak

[4] S. Parrott, {\it On a Quotient Norm and the Sz.-Nagy - Foia\c{s} Lifting Theorem}, Journal of Functional Analysis {\bf 30}, 311-328 (1978). 

\medbreak 

[5] F. Riesz and B. Sz.-Nagy, {\it Functional Analysis}, Frederick Ungar Publishing (1955); Dover Publications (1990). 

\medbreak 

[6] D. Sarason, {\it New Hilbert Spaces from Old}, in {\it Paul Halmos - Celebrating 50 Years of Mathematics}, 195-204, Springer-Verlag (1991). 

\medbreak 

[7] D. Sarason, {\it Sub-Hardy Hilbert Spaces in the Unit Disc}, University of Arkansas Lecture Notes in the Mathematical Sciences {\bf 10}, Wiley-Interscience (1994). 

\medbreak

[8] J.J. Sch\"affer, {\it On unitary dilations of contractions}, Proceedings of the American Mathematical Society {\bf 6}, 322 (1955). 

\medbreak 

[9] B. Sz.-Nagy, {\it Sur les contractions de l'espace de Hilbert}, Acta Scientiarum Mathematicarum {\bf 15}, 87-92 (1953). 

\medbreak 

[10] N. Young, {\it An introduction to Hilbert space}, Cambridge Mathematical Textbooks, Cambridge University Press (1988).

\medbreak

\end{document}